\documentclass[11pt]{amsart}
\usepackage{rotating}
\usepackage{amssymb, amsmath}
\allowdisplaybreaks
\usepackage[all]{xy}
\usepackage[inline]{enumitem}
\usepackage{hyphenat}
\usepackage[colorlinks,linkcolor=blue,citecolor=blue,urlcolor=teal]{hyperref}
\usepackage{xcolor}
\usepackage{mathtools}
\usepackage{tikz-cd}
\usepackage[labelformat=empty]{caption}
\usepackage{cleveref} 

\usepackage[ left=3cm, right=3cm, top = 2.8cm, bottom = 2.8cm ]{geometry}

\newcommand\numberthis{\addtocounter{equation}{1}\tag{\theequation}}

\newcommand{\eq}[2]{\begin{equation}\label{#1}#2 \end{equation}}

\newcommand{\ra}{\rightarrow} 
\newcommand{\xra}{\xrightarrow}

\newcommand{\N}{\mathbb{N}}

\newcommand{\Z}{\mathbb{Z}}

\newcommand{\cA}{\mathcal{A}} 
\newcommand{\cC}{\mathcal{C}} 

\newcommand{\Set}{{\operatorname{\mathrm{Set}}}}

\newcommand{\Ab}{\operatorname{\mathrm{Ab}}}

\newcommand{\Hom}{\operatorname{Hom}}

\newcommand{\Ker}{\operatorname{Ker}}

\renewcommand{\Im}{\operatorname{Im}}

\newcommand{\Mod}{\operatorname{Mod}}

\newcommand{\Top}{\operatorname{Top}}
\newcommand{\bbox}{{\blacksquare}}
\newcommand{\Prof}{\operatorname{Prof}}

\newcommand{\cg}{\operatorname{cg}}
\newcommand{\De}{\operatorname{De}}
\newcommand{\CS}{\operatorname{CS}}
\newcommand{\cgwh}{\operatorname{cgwh}}
\newcommand{\chaus}{\operatorname{chaus}}
\newcommand{\qs}{\operatorname{qs}}

\newcommand{\ML}{\operatorname{ML}}

\newcommand{\TopMod}{\operatorname{TopMod}}
\newcommand{\TopAb}{\operatorname{TopAb}}

\newcommand{\op}{{\operatorname{op}}}

\renewcommand{\lim}{\operatornamewithlimits{\rm{lim}}}
\newcommand{\colim}{\operatornamewithlimits{\rm{colim}}}

\newcommand{\LCA}{\operatorname{LCA}}

\newcommand{\fp}{\operatorname{{fp}}}

\newcommand{\Spec}{\operatorname{Spec}}
\newcommand{\Spa}{\operatorname{Spa}}
\newcommand{\Proj}{\operatorname{Proj}}

\newcommand{\Cond}{\operatorname{Cond}}
\newcommand{\Solid}{\operatorname{Solid}}

\newcommand{\hra}{\hookrightarrow}

\renewcommand{\hat}{\widehat} 

\renewcommand{\phi}{{\varphi}}

\newcommand{\ul}{\underline}

\renewcommand{\bar}{\overline}

\newcommand{\Pro}{\operatorname{Pro}}
\newcommand{\Ind}{\operatorname{Ind}}

\theoremstyle{plain}

\newtheorem{prop}{Proposition}[section]
\newtheorem{lem}[prop]{Lemma}

\newtheorem{thm}[prop]{Theorem}

\theoremstyle{definition}
\newtheorem{defn}[prop]{Definition}

\newtheorem{para}[prop]{}

\newtheorem{rmk}[prop]{Remark}

\newcommand{\beq}{\begin{equation}}
\newcommand{\eeq}{\end{equation}}

\numberwithin{equation}{prop}

\title{Notes on Coherent six-functor formalisms: Pro vs Solid}
\author{Fei Ren}

\address{Bergische Universit\"at Wuppertal\\ Gau\ss str. 20, D-42119 Wuppertal, Germany}
\email{renfei@uni-wuppertal.de}

\thanks{ }
\begin{document}
\sloppy

\begin{abstract}
In the classical theory for coherent sheaves, the only missing piece in the Grothendieck six-functor formalism picture is $j_!$ for an open immersion $j$.
Towards fixing this gap,
classically Deligne proposed a construction of $j_!$ by extending the sheaf class to pro sheaves, while Clausen-Scholze provided another solution by extending the sheaf class to solid modules.
In this work, we prove that Deligne's construction coincides with the Clausen-Scholze construction via a natural functor, whose restriction to the full subcategory of Mittag-Leffler pro-systems is fully faithful.
\end{abstract}
\subjclass{14F08} 
\maketitle

\tableofcontents

\section{Introduction}

The Grothendieck six-functor formalism plays a foundational role in modern algebraic geometry, providing a conceptual framework for understanding the interplay between various operations on sheaves across different geometric contexts. 
While this formalism is fully developed in the context of $\ell$-adic constructible sheaves in the étale cohomology in the 1970s, a parallel formalism for coherent sheaves remains incomplete. 
In particular, for an open immersion 
$j: U\hra X$, the exceptional pushforward functor $j_!$
does not exist as a functor between categories of (quasi)coherent sheaves.

There have been various attempts towards building this missing $j_!$ via enlarging the category of coherent sheaves.  
In \cite[Appendix]{Hartshorne_RD}, Deligne constructed a version of $j_!$
by embedding coherent sheaves into the larger category of pro-coherent sheaves. 
More recently, Clausen and Scholze introduced the framework of solid modules, where the enlargement of  the target category allows for a robust treatment of $j_!$ in a derived setting, see \cite[Lecture VIII]{CS}. 
Their theory of solid modules offers a powerful and flexible toolset to do certain computations, yet its relation to the earlier construction of Deligne has remained unclear.

The aim of this paper is to bridge these two constructions. 
We prove that Deligne’s construction of 
$j_!$ via pro-sheaves coincides with the Clausen–Scholze version defined via solid modules, through a natural comparison functor $R\Phi$. 
Moreover, $R\Phi$ is the derived functor of $\Phi$, whose restriction to
the subcategory of Mittag-Leffler pro-systems is fully faithful. This result provides a conceptual and structural unification of the two approaches and reinforces the compatibility of classical techniques with modern condensed tools.
\begin{thm}[{\Cref{thm1}, \Cref{thm2}}]
Let $A$ be a commutative ring with $1$.
\begin{enumerate} 
\item 
There is a natural functor 
$$\Phi:\Pro_{\N}\Mod(A)\ra \Solid(A), \quad ``\lim_\N" M_i\mapsto \ul{\lim_\N M_i}$$
from the category $\Pro_{\N}\Mod(A)$ of $\N^\op$-indexed pro-$A$-modules to the category $\Solid(A)$ of solid $A$-modules.
The functor $\Phi$ is  left exact, and its restriction to the full subcategory of Mittag-Leffler pro-$A$-modules
$$\Phi^{\ML}:\Pro_{\N}^{\ML}\Mod(A)\ra \Solid(A)$$
is exact and fully faithful.
\item 
Deligne's construction for $j_!$ is compatible with the one built by Clausen-Scholze via the right derived functor $R\Phi$ of $\Phi$. (See \Cref{thm2} for the precise statement.)
\end{enumerate}
\end{thm}
This comparison result shows the foundational coherence of these theories. 
Although in the current stage, the theorem is only stated for 
a special type of open immersions, we believe that it gives an interesting idea.
We expect the theorem remains true for more general open immersion $j$, and we plan to glue the functor $R\Phi$ in the forthcoming work.

\subsection{Notations and first descriptions}
Let $A$ be a discrete commutative ring with $1$.
Let $\cC$ be a category. Let $\kappa$ be an uncountable strong limit cardinal.
\begin{itemize}
\item 
$\Top$: the category of topological spaces.
\item 
$\Top^{\cg}$: the full subcategory of $\Top$ consisting of compactly generated topological spaces.
\item 
$\Top^{\cgwh}$: the full subcategory of $\Top$ consisting of compactly generated weakly Hausdorff spaces.
\begin{itemize}
\item 
It is complete and cocomplete \cite[Corollary 2.23, Proposition 2.30]{Strickland_CGHW}.
The limits in $\Top^{\cgwh}$ are given by the k-ification of the corresponding limits in $\Top$.
The colimits in $\Top^{\cgwh}$ are given by a quotient space of the corresponding colimits in $\Top$.
\end{itemize}
\item 
$\Ab$: The category of abelian groups.
\item 
$\TopAb$: the category of topological abelian groups.
\item 
$\TopAb^\delta$: the category of discrete abelian groups.
It is canonically equivalent to $\Ab$.
\item 
$\TopAb^{\chaus}$: the category of topological abelian groups whose underlying topological spaces are compact 
\footnote{For us, the term ``\textit{compact}" does not require being Hausdorff.}
Hausdorff.
\item 
$\Mod(A)$: the category of $A$-modules.
\item 
$\Mod(A)^{\fp}$: the category of finitely presented $A$-modules.

\item 
$\TopMod(A)$: the category of topological $A$-modules.
\item 
$\TopMod(A)^{\cgwh}$: the full subcategory of $\TopMod(A)$ consisting of those objects with compactly generated weakly Hausdorff underlying topological spaces.
\item 
$\TopMod(A)^\delta$: the full subcategory of $\TopMod(A)$ consisting of those objects with discrete underlying topological spaces.
It is canonically equivalent to $\Mod(A)$.
\item 
$\LCA$: the category of locally compact Hausdorff abelian groups. 
\item 
$\Pro\cC$: the category of pro-objects in $\cC$ (indexed by any filtered category). 
We follow the same definition as in \cite[\S A.2]{AM}. 
\begin{itemize}
\item 
By \cite[A. Corollary 3.2]{AM}, any morphism between two objects $``\lim_i"X_i,``\lim_j"Y_j$ in $\Pro\cC$ can be represented by a levelwise morphism with the same index category.
\item 
The Yoneda functor identifies $\Pro\cC$ as a full subcategory of the functor category ${\rm Fun}(\cC,\Set)^{\op}$. 
\end{itemize}
\item 
$\Pro^{\ML}\cC$, when $\cC$ admits kernels and cokernels (so that images are defined): the full subcategory of $\Pro\cC$ consisting of pro-systems satisfying the Mittag-Leffler condition. 
Namely, a pro-system $``\lim_I X_i"\in\Pro\cC$ lies in $\Pro^{\ML}\cC$ if and only if
for any $i\in I$, there exists a $j\in I$ such that for all $l\ge j$,
$\Im(\phi_{li})=\Im(\phi_{ji})$ holds. ($\phi_{ji}:X_j\ra X_i$ refers to the structure maps in the pro-system.)
\item 
$\Pro_{\N}\cC$: the full subcategory of $\Pro\cC$ generated by $\N^\op$-indexed pro-objects.
\begin{itemize}
\item 
By possibly interleaving the indices of $j\in \N$, any morphism between two objects $``\lim_i"X_i,``\lim_j"Y_j$ in $\Pro_{\N}\cC$ is represented by a levelwise morphism, i.e., an object in $\Pro_{\N}\Hom_{\cC}(X_i,Y_i)$
\end{itemize}
\item 
$\Ind\cC$: the category of ind-objects in $\cC$. By definition, it is $(\Pro\cC^{\op})^{\op}.$
\item 
$\Ind_\N\cC$: the full subcategory of $\Ind\cC$ generated by $\N$-indexed ind-objects.
\item 
$\Proj_\kappa$: the site of extremally  disconnected profinite sets of size $<\kappa$ with covers given by finite families of jointly surjective maps. 
\item 
$\Prof_\kappa$: the site of profinite sets of size $<\kappa$ with covers given by finite families of jointly surjective maps. 
\item 
$\Cond:=\colim_\kappa \Cond_\kappa$: the category of condensed sets as defined in \cite[p. 15]{CS}.
Similarly, 
$\Cond(A):=\colim_\kappa \Cond_\kappa(A)$ is the category of condensed $A$-modules, etc.
\item 
$\Cond^{\qs}$, $\Cond(A)^{\qs}$, etc.: the full subcategory of condensed sets / $A$-modules / etc.  whose underlying condensed set is quasiseparated.
\item 
$\Solid(A)$ (where $A$ is a finite type $\Z$-algebra): the abelian category of solid $A$-modules as defined in \cite[Proposition 7.5(i)]{CS} under the notation $\Mod_A^{\rm cond}$. It is a full subcategory of $\Cond(A)$.
\item 
$D(\Solid(A))$ or $D(A_\bbox)$ (where $A$ is a finite type $\Z$-algebra): the derived category of solid $A$-modules.
\item 
$D(\cA)$ (where $\cA$ is an analytic ring as defined in \cite[Definition 7.4]{CS}): the derived category of solid $\cA$-modules.
\end{itemize}

\subsection{Acknowledgments}
The author thanks Andreas Bode for many intense discussions on this topic and for providing detailed feedback on an earlier draft of this paper.
He is grateful to Kay Rülling for helpful discussions and for pointing out an error in an earlier version.
Thanks are also due to Oliver Braunling for several conversations regarding general topological groups, and for pointing out the references \cite{Prosmans} and \cite{Wilson}.
In addition, the author thanks Chirantan Chowdhury, Jorge Galindo Pastor, Joshua Mundinger, Mikhail Tkatchenko, and Matthias Wendt for helpful email exchanges.
Finally, the author is highly grateful to the anonymous referee for a detailed report, which provided the opportunity to clarify and carry out many of the detailed arguments.

\section{Full faithfulness}

\begin{defn}
Define a natural functor
\eq{Phi1}{\Phi_1:\Pro_\N\Mod(A)\xra{} \TopMod(A),\quad``\lim " M_i\ra \lim M_i.}
Here the $\lim M_i$ refers to the limit in the category $\TopMod(A)$, i.e., the inverse limit of $M_i$ as modules equipped with the inverse limit topology of discrete spaces.
\end{defn}

\begin{prop}\label{Phi1FF}
The following restriction of $\Phi_1$
\eq{{Phi1ML}}{
\Pro^{\ML}_{\N}\Mod(A)\subset \Pro_\N\Mod(A)\xra{\Phi_1} \TopMod(A),\quad``\lim " M_i\ra \lim M_i}
is fully faithful.
\end{prop}

\begin{proof}
We prove something slightly stronger:
Let $``\lim_{i\in \N}" M_i\in \Pro^{\ML}_{\N}\Mod(A), ``\lim_J" N_j\in \Pro\Mod(A)$, the category of pro-$A$-modules with any index set. Then
\begin{align*} 
\Hom_{\Pro\Mod(A)}(``\lim_{i\in \N}" M_i,``\lim_J"N_j)
&\stackrel{(1)}{=}\lim_J \colim_{i\in \N} \Hom_{\Mod(A)}(M_i,N_j)
\\
&\stackrel{(2)}{=}\lim_J \Hom_{\TopMod(A)}(\lim_{i\in \N} M_i,N_j)\\
&\stackrel{(3)}{=}\Hom_{\TopMod(A)}(\lim_{i\in \N} M_i,\lim_J N_j).
\end{align*}
Here 
(1) follows from the definition of the pro-category.     

(2) 
We show that the canonical map 
$$\colim_{i\in \N} \Hom_{\Mod(A)}(M_i, N_j) \to \Hom_{\TopMod(A)}(\lim_{i\in \N} M_i, N_j)$$
is a bijection.

\textit{Surjectivity}.
Let $f: \lim_{i\in \N} M_i \to N_j$ be a continuous map. Since the target space $N_j$ is discrete, $\{0\}$ is open, so $\Ker(f)$ is an open neighborhood of $0$ in $\lim_{i\in \N} M_i$. 
By the definition of the inverse limit topology, a basic open neighborhood of $0$ is a finite intersection of sets of the form $\Ker(p_i)$. Thus, there exists a finite subset $I' \subset \N$ such that $\bigcap_{i \in I'} \Ker(p_i) \subset \Ker(f)$. Because $\N$ is totally ordered, $I'$ has a maximum element $k$. For any $i \in I'$, $p_i$ factors through $p_k$, meaning $\Ker(p_k) \subset \Ker(p_i)$. Thus $\Ker(p_k) \subset \Ker(f)$.
This inclusion means $f$ factors through the image of $p_k$, yielding a well-defined map 
$$\bar{f}: \Im(p_k) \to N_j.$$
Next, we use the Mittag-Leffler condition. Because $``\lim_{i\in \N}" M_i$ is Mittag-Leffler, there exists an index $l \ge k$ such that for all $m \ge l$, $\Im(\phi_{mk}) = \Im(\phi_{lk})$. 
The Mittag-Leffler condition, together with the countable index set, guarantees that the 
projection map $p_k:\lim_{i\in \N} M_i\ra M_k$ maps
exactly onto this stable image, namely 
$$\Im(p_k) = \Im(\phi_{lk}).$$
Therefore, our map $\bar{f}$ can be rewritten as
$$\bar f: \Im(\phi_{lk}) \ra N_j.$$
Define
$$f_l = \bar{f} \circ \phi_{lk}: M_l \to N_j.$$
This $f_l\in \colim_{i\in \N} \Hom_{\Mod(A)}(M_i, N_j)$ clearly maps to $f$ under the canonical map, proving surjectivity.

\textit{Injectivity.}
Suppose two maps $g_1, g_2: M_k \to N_j$ have the same image in $\Hom_{\TopMod(A)}(\lim_{i\in \N} M_i, N_j)$. Let $g = g_1 - g_2$. Then 
$$g \circ p_k = 0.$$
By the same Mittag-Leffler stabilization argument, there exists an $l \ge k$ such that $\Im(\phi_{lk}) = \Im(p_k)$. 
Therefore, $g \circ \phi_{lk} = 0$, which means 
$$g_1 \circ \phi_{lk} = g_2 \circ \phi_{lk}.$$
This implies $g_1$ and $g_2$ define the same element in the colimit $\colim_{i\in \N} \Hom_{\Mod(A)}(M_i, N_j)$, proving injectivity.

(3)
follows from the universal property of $\lim_J$.

\end{proof}

\begin{rmk}
\begin{enumerate}
    \item 
One might want such a functor to land in $\LCA$, but  $\Phi_1$ does not satisfy this property. 
The key point is that the limit of locally compact spaces in $\Top$ are not necessarily locally compact.
For example, if we equip $\prod_\N \Z$ with the Tychonoff product topology, then this is a limit of locally compact abelian groups (i.e., $\prod^n_{i=1} \Z$), but itself is not locally compact, as a countable direct product $\prod G_n$ of locally compact $G_n$'s is locally compact if and only if  all $G_n$ are compact except for finitely many $n$.
\item 
If $I$ is an uncountable totally ordered set, an inverse system can satisfy the Mittag-Leffler condition while the big enough projection map does not surject to the stable image of transition maps. 
We thank an anonymous referee for pointing this out.

An example is the following.
Let $I = \omega_1$, the totally ordered set of countable ordinals. 
For each $\alpha < \omega_1$, define the set $B_\alpha$ as follows:
$$ B_\alpha = \{ f: \alpha \to \mathbb{N} \mid f \text{ is injective and } \mathbb{N} \setminus \Im(f) \text{ is infinite} \} $$
Since every $\alpha < \omega_1$ is countable by definition, such injective functions always exist, so $B_\alpha$ is non-empty. 
For $\alpha \le \beta < \omega_1$, the transition map 
$$\pi_{\beta\alpha}: B_\beta \to B_\alpha$$
is simply the restriction $f \mapsto f|_\alpha$.
Such a $\pi_{\beta\alpha}$ is always surjective: 
Because the complement of the image of any $f \in B_\alpha$ is infinite, 
we can always extend $f$ injectively to the countable set $\beta \setminus \alpha$ while still leaving an infinite complement. 
In particular, this system $\{B_\alpha\}_{\alpha<\omega_1}$ is Mittag-Leffler.

Now, let $M_\alpha = \mathbb{Z}^{(B_\alpha)}$ be the free abelian group generated by the set $B_\alpha$. 
The transition maps 
$$\phi_{\beta\alpha}: M_\beta \to M_\alpha$$
are the linear extensions of $\pi_{\beta\alpha}$,.
So they are also surjective, which means
$\Im(\phi_{\beta \alpha}) = M_\alpha$ for all $\beta \ge \alpha$. 
Hence, this system $\{M_\alpha\}_{\alpha<\omega_1}$ is Mittag-Leffler.

We show that 
$$\lim_{\omega_1} M_\alpha = 0.$$
Suppose there is a non-zero element $x \in \lim_{\omega_1} M_\alpha$. 
Since $x$ is non-zero and $I$ is totally ordered, there exists an $\alpha<\omega_1$ such that for all $\beta\ge \alpha$, the projection $x_\beta$ of $x$ in $M_\beta$  are nonzero. 
In a free abelian group, any element $x_\alpha$ is a finite linear combination of basis elements. Let 
$$S_\beta:=\{ f\in B_\beta\mid 
\text{coefficient of $f$ in $x_\beta$ is nonzero}
\}.$$
$S_\beta$ is 
\begin{itemize} 
\item 
nonempty for big enough $\beta$, because $x_\beta$ is nonzero as long as $\beta\ge \alpha$; and 
\item 
finite, because in a free abelian group, any element $x_\alpha$ is a finite linear combination of basis elements. 
\end{itemize}
Because the transition maps of 
$\{M_\alpha\}_{\alpha<\omega_1}$
are linear and send basis elements to basis elements, 
the map $\pi_{\beta\alpha}$ with $\beta\ge \alpha$ must map $S_\beta$ surjectively onto $S_\alpha$. 
This makes
$$\{S_\beta\}_{\beta<\omega_1}$$ 
an inverse system of non-empty finite discrete spaces with surjective transition maps. 
By \cite[Proposition 1.1.4]{RZ2010}, its inverse limit is non-empty. 
This means there exists a compatible family of basis elements, giving us a well-defined function $$F: \omega_1 \to \mathbb{N}$$
that is injective on every countable ordinal, meaning $F$ is itself injective. 
But it is impossible to inject the uncountable set $\omega_1$ into the countable set $\mathbb{N}$. 
This contradiction proves that $x$ must be $0$, so $\lim_{\omega_1} M_\alpha = 0$.

For this example, the image  $\Im(p_\beta) = \{0\}$ for $\beta$ big enough, whereas the Mittag-Leffler stable image $\Im(\pi_{\beta\alpha}) = M_k \neq \{0\}$. 

\end{enumerate}
\end{rmk}

\begin{prop}\label{Phi1landsincgwh}
The functor $$\Phi_1:\Pro_{\N}\Mod(A)\xra{\eqref{Phi1}} \TopMod(A)$$ lands in  $\Top\Mod(A)^{\cgwh}$.
\end{prop}

\begin{proof}
Any prodiscrete $A$-module with the index category $\N$ is metrizable (and hence compactly generated and weakly Hausdorff).
In fact, as any subspace of a metrizable space is metrizable, it suffices to show that any $\N$-indexed product of discrete spaces is metrizable.
This holds because every discrete space is trivially metrizable, and a countable product of metrizable spaces is metrizable by \cite[IX. Corollary 7.3]{Dugundji66}.
\end{proof}

\begin{rmk}
The proposition is false for general pro-systems (as opposed to $\N^\op$-indexed pro-systems). 
The key point is that prodiscrete spaces are not necessarily compactly generated. 
For example, $\prod_{\aleph_1} \Z$ (where $\aleph_1$ is the first uncountable cardinal) is not compactly generated.
\end{rmk}

\begin{lem}\label{liminProNC}
\begin{enumerate}
\item 
For any category $\cC$, $\Pro_{\N}\cC$ admits $\N$-indexed limits.
\item 
If $\cC$ admits all finite limits (resp. finite colimits), then so does $\Pro_{\N}\cC$.
\item 
If $\cC$ is an abelian category, then so is $\Pro_{\N}\cC$.
\end{enumerate}

\end{lem}
\begin{proof}
\begin{enumerate}
\item 
Let $``\lim_{i\in \N}"X^i\in \Pro_{\N}\Pro_{\N}\cC$, where for each $i\in \N$, $X^i:=``\lim_{j\in \N}"X^i_j\in \Pro_{\N}\cC.$
We use the fact that the diagonal $\N\ra \N\times \N$ is cofinal.
Then $$``\lim_{(i,j)\in \N\times \N}"X^i_j$$ is the limit of $``\lim_{i\in \N}"X^i$ in $\Pro\cC$ by \cite[A. Proposition 4.4]{AM}.
In particular, it is the limit of $``\lim_{i\in \N}"X^i$ in the full subcategory $\Pro_{\N}\cC$.
\item 
Suppose $\cC$ admits all finite limits. To show that $\Pro_{\N} \cC$ admits all finite limits, it suffices to show that it admits a terminal object and pullbacks. 
The terminal object is simply the constant pro-system at the terminal object of $\cC$.

For pullbacks, we simply apply the construction in \cite[Proposition 4.2]{AM}.
Let $X \to Z \leftarrow Y$ be a diagram in $\Pro_{\N} \cC$. 
Since the index category $(\bullet \to \bullet \leftarrow \bullet)$ is finite and has no loops, 
we can apply the uniform approximation theorem \cite[A. Proposition 3.3]{AM} to straighten the diagram.
That is, it is isomorphic to a pro-object of diagrams in $\mathcal{C}$, say $``\lim_{e\in E}"(X_e \to Z_e \leftarrow Y_e)$, 
indexed by a filtering category $E$. 
Because $\mathcal{C}$ admits pullbacks, 
we can take the levelwise pullback $ X_e \times_{Z_e} Y_e$ in $\mathcal{C}$. 
The resulting pro-object $``\lim_{e\in E}" X_e \times_{Z_e} Y_e$,
by  \cite[Proposition 4.1]{AM},
is the pullback $X \times_Z Y$ in $\Pro\mathcal{C}$. 
Finally, since the original pro-objects $X, Y, Z$ are indexed by $\mathbb{N}^{\text{op}}$, 
the index category $E$ is a countable cofiltered category. 
Because any countable cofiltered category admits a cofinal functor from $\N^{\op}$, the pullback $W$ is isomorphic to a pro-object indexed by $\mathbb{N}^{\text{op}}$. 
Hence, the pullback lies in $\text{Pro}_{\mathbb{N}}\mathcal{C}$. 

The construction for finite colimits is similar. 
\item 
Since finite limits and finite colimits are constructed levelwise, we can check the axioms of an abelian category from the constructions given in (2).
\end{enumerate}
\end{proof}

\begin{prop}\label{Phi1commlim}
$\Phi_1:\Pro_{\N}\Mod(A)\ra \TopMod(A)^{\cgwh}$ commutes with $\N$-indexed limits and finite limits.
\end{prop}
\begin{proof}
By \Cref{liminProNC},
$\mathbb{N}$-indexed limits and finite limits in $\text{Pro}_{\mathbb{N}}\text{Mod}(A)$ can be constructed via diagonal systems and levelwise limits, respectively. Applying the functor $\Phi_1$ amounts to taking the inverse limit in $\text{TopMod}(A)^{\text{cgwh}}$. Since limits commute with limits in the category of topological modules, $\Phi_1$ naturally commutes with these limits. 
\end{proof}

\begin{para}\label{Phi2FF}
By {\cite[Theorem 2.16(ii)]{CS}}, the functor
\eq{ul}{\TopMod(A)^{\cgwh}\xra{} \Cond(A)^{\qs},\quad M\mapsto \ul M:=\Hom_{\Top}(-,M)}
is well-defined and fully faithful.
Moreover, it is a right adjoint and hence commutes with all limits.
By \cite[Lemma 4.14]{CS2}, the inclusion functor 
$$\Cond(A)^{\qs}\hra \Cond(A)$$ is a right adjoint and hence commutes with all limits.
We write
\eq{Phi2FFeq1}{
\Phi_2:\TopMod(A)^{\cgwh}\xra{} \Cond(A)
}
the composite functor of \eqref{ul} and $\Cond(A)^{\qs}\hra \Cond(A)$.
$\Phi_2$ is therefore fully faithful and commutes with all limits.
\end{para}
\begin{rmk}
Before we proceed, we summarize some related adjunctions in the diagram below:
\[
\xymatrix@C=5em{
\Top
\ar@<1ex>@{<-^)}[r]^{\text{\cite[p.8]{CS}}}
  \ar@{}[r]|{\perp}
  \ar@<-1ex>@{->}[r]_{\rm k-ification} 
  &
\mathrm{Top}^{\mathrm{cg}} 
\ar@<1ex>@{->}[r]^{\text{ {\cite[Prop. 2.22]{Strickland_CGHW}}}} 
\ar@{}[r]|{\perp}
\ar@<-1ex>@{<-_)}[r]
 &
\mathrm{Top}^{\mathrm{cgwh}}  
\ar@{}[d]|{\dashv}
\ar@<1ex>@{^(->}[d]^{
\text{{\cite[Th. 2.16(ii)]{CS}}}
}
\ar@<-1ex>@{<-}[d]
\\&
\Cond
\ar@<1ex>@{->}[r]
\ar@{}[r]|{\perp}
\ar@<-1ex>@{<-_)}[r]_{\text{\cite[Lem. 4.14]{CS2}}} 
&
\Cond^{\qs}
}
\]
Moreover, the forgetful functors forgetting the algebraic structure (e.g. $\TopAb\ra \Top$, $\TopMod(A)\ra \Top$, etc.) are typically right adjoint and hence commutes with all limits.
\end{rmk}

\begin{lem}\label{ColimCok}
Let $\alpha:M\ra N$ be a strict epimorphism in $\TopMod(A)$, namely, it is  a continuous, surjective homomorphism that is also an open map.
If the underlying topological space of $M$ is 
\begin{enumerate} 
\item 
sequentially complete (i.e., every Cauchy sequence admits a limit point) and $T_1$ (i.e., any singleton is closed)
\item 
first countable (i.e., possesses a countable neighborhood basis at every point),
\end{enumerate} 
then 
$$\ul\alpha: \ul M\ra \ul N$$
is an epimorphism in $\Cond(A).$
\end{lem}

\begin{proof}
It suffices to show that the underlying map of condensed sets $\ul M \ra \ul N$ is an epimorphism. 
This means that for any profinite set $S$ and any continuous map $f: S \ra N$, 
there exists a profinite set $S'$ and a continuous surjection $\pi: S' \ra S$ 
such that the composition $f \circ \pi$ lifts to a continuous map $g: S' \ra M$.

We construct $S'$ via an inverse limit of finite clopen covers. 
To this end, we first fix once and for all a basis of symmetric open neighborhoods 
$$W_0 \supset W_1 \supset W_2 \dots$$
of $0 \in M$, such that  
$$W_{n+1} + W_{n+1} \subset W_n.$$
This is possible because $M$ is first countable.
Consider the fiber product $E = S \times_N M$ in the category of topological spaces. 
Because $M \ra N$ is an open surjection, and both being open and being surjective are properties stable under base change, the projection map $p: E \ra S$ is also an open surjection. 
Let 
\begin{center}
$S_0 = S$\qquad  and \qquad $E_0 = E$. 
\end{center}
Since $p$ is an open surjection, for each $s \in S_0$, we can choose
\begin{itemize} 
\item 
a point $e \in E_0$ such that $p(e) = s$, and 
\item 
an open neighborhood $U_{0,e} \subset E_0$ of $e$ such that its projection $U_{0,M}$ to $M$ is contained in $x_0+W_0$, where $x_0\in M$ is the projection of $e$ to $M$.
\end{itemize} 
Because $p$ is an open map, $p(U_0)$ is an open neighborhood of $s$ in $S_0$. 
Since $S_0$ is totally disconnected,  the clopen sets form a basis for the topology. So we can choose 
\begin{itemize} 
\item 
a clopen neighborhood $V_{0,s} \subset p(U_{0,e})$ of $s$. 
\end{itemize}
By the compactness of $S_0$, a finite number of these clopens, 
$\{V_{0,s_i}\}_{i\in I_0}$, cover $S_0$.
(We call these chosen points $s_i$ ($i\in I_0$) the \textit{center points} in step 0.)
We can disjointly put them together to obtain
$$S_1 := \bigsqcup_{i \in I_0} V_{0,s_i}.$$
The natural map $S_1 \ra S_0$ is a continuous surjection from a profinite set. 
Write $e_i$ our choice of the point in $E_0$ which lies over $s_i$.
Let 
$$E_1 := \bigsqcup_{i \in I_0} (U_{0,e_i} \cap p^{-1}(V_{0,s_i})).$$
The projection $p_1: E_1 \ra S_1$ is again an open surjection.
These maps fit into a commutative diagram
$$\xymatrix{
E_1\ar[r]\ar@{->>}[d]_{p_1}&E_0\ar@{->>}[d]_p\ar[r]&M\ar@{->>}[d]^\alpha\\
S_1\ar@{->>}[r]&S_0\ar[r]^f&N.
}$$
At step $n$, for each $s\in S_n$, we choose 
\begin{itemize}
\item 
a point $e\in E_n$ such that $p(e)=s$,
\item 
an open neighborhood $U_{n,e}\subset E_n$ of $e$ such that its projection $U_{n,M}$ to $M$ is contained in $x_n + W_n$, where $x_n \in M$ is the projection of $e$ to $M$, and
\item 
a clopen neighborhood $V_{n,s}$ of $s$ such that $V_{n,s}\subset p_n(U_{n,e})$.
\end{itemize}
Then we construct
$$S_{n+1}:=\bigsqcup_{i \in I_n} V_{n,s_i},$$
where $\{V_{n,s_i}\}_{i\in I_n}$ is a finite open cover of $S_n$,
and 
$$E_{n+1}:= \bigsqcup_{i \in I_n} (U_{n,e_i} \cap p_n^{-1}(V_{n,s_i})).$$
(We call these chosen points $s_i$ ($i\in I_n$) the \textit{center points} in step $n$.)
Iterating this process inductively gives an inductive system of profinite sets
$$\dots \ra S_2 \ra S_1 \ra S_0 = S,$$ 
and open surjections $p_n: E_n \ra S_n$, which fit in a commutative diagram of the form
$$\xymatrix{
\dots\ar[r]&E_n\ar[d]_{p_n}\ar[r]&\dots\ar[r]&E_1\ar[r]\ar@{->>}[d]_{p_1}&E_0\ar@{->>}[d]_p\ar[r]&M\ar@{->>}[d]^\alpha
\\
\dots\ar[r]&S_n\ar[r]&\dots\ar[r]&S_1\ar@{->>}[r]&S_0\ar[r]^f&N.
}$$
Let $$S' = \lim_n S_n.$$
As an inverse limit of profinite spaces, $S'$ is a profinite set itself.
The induced projection $\pi: S' \ra S$ is a continuous surjection.

We define the map $g:S'\ra M$.
Let $s'=(s_0, s_1, s_2, \dots)\in S' = \lim_n S_n$.
By construction, $s_n$ is not necessarily a center point in  step $n$;
while with the help of $s_{n+1} \in S_{n+1} =
\bigsqcup_{t \in I_n} V_{n,t}$, 
we can find exactly one disjoint component, say 
$V_{n,t^\circ}$,  
containing $s_{n+1}$.
We write $s_n^\circ:=t^\circ.$
This $s_n^{\circ}$ is by construction one of the center points in step $n$.
Having found a center point, the construction further provides us
\begin{itemize}
\item 
a chosen point $e_n^{\circ}\in E_n$ mapping down to $s_n^{\circ}$ via $p_n$,
\item 
an open neighborhood $U_n$ of $e_n^{\circ}$ whose projection $U_{n,M}$ to $M$ is contained in  $x_n+W_n$, where $x_n$ is the projection of $e_n^\circ$ to $M$, and 
\item 
a clopen neighborhood $V_n$ of $s_n^\circ$ in $S_n$ such that $V_{n}\subset p_n(U_n)$.
(This $V_n$ is precisely $V_{n,s_n^\circ}$.)
\end{itemize} 
By the construction of $E_{n+1}$, the point $e^{\circ}_{{n+1}} \in E_{n+1}$ belongs to the component $U_{n} \cap p_n^{-1}(V_{n})$. 
Therefore, the projection of $e_{n+1}^\circ$ to $E_n$ lies in $U_{n}$.  
Because the projection $E_{n+1}\ra M$ factors through the projection $E_n\ra M$, the image of $e_{n+1}^\circ$ in $M$, which is exactly $x_{n+1}$, lies in $U_{n,M}$, the projection of $U_{n}$ to $M$. 
By our choice of $U_{n}$, this projection $U_{n,M}$ is contained in $x_n + W_n$.
This forces:
$$x_{n+1} \in x_n + W_n \implies x_{n+1} - x_n \in W_n.$$
For any $m > n \ge 2$, we have 
\eq{ColimCok_Cauchy}{x_m - x_n = \sum_{k=n}^{m-1} (x_{k+1} - x_k) \in W_n + W_{n+1} + \dots + W_{m-1} \subset W_{n-1}.}
Therefore the sequence $(x_n)$  is a Cauchy sequence in $M$. 
Since $M$ is sequentially complete and $T_1$, this Cauchy sequence converges to a unique limit point $x \in M$. 
We define the map $g:S'\ra M$ by mapping $s'$ to this limit point $x$.

The map $g:S'\ra M$ so constructed is indeed a lift of $f\circ \pi: S'\ra N$ via $\alpha:M\ra N$, namely, the diagram
$$\xymatrix{
&&M\ar[d]^\alpha\\
S'\ar[r]_\pi\ar[rru]^g&S\ar[r]_f&N
}$$
commutes.
To see this, let us put everything together and label the maps that we are going to use afterwards:
\vspace{3cm}
$$
\xymatrix{
&&\dots\ar[r]&E_n\ar[d]_{p_n}\ar[r]^{}&
\dots\ar[r]&E_1\ar[r]^{}\ar@{->>}[d]_{p_1}&E_0\ar@{->>}[d]_p\ar[r]^{f_M}&M\ar@{->>}[d]^\alpha
\\
S'\ar@/_3ex/[rrr]\ar@/_5ex/[rrrrr]\ar@/_6ex/[rrrrrr]_(.8){\pi}
\ar@/^20ex/[urrrrrrr]^g
\ar@/^5ex/[urrr]_{g_n}\ar@/^11ex/[urrrrr]_{g_1}\ar@/^14ex/[urrrrrr]_{g_0}
&&\dots\ar[r]&S_n\ar[r]^{}&\dots\ar[r]&S_1\ar@{->>}[r]^{}&S_0\ar[r]^f&N.
}\\[5ex]$$
Let $s'=(s_0, s_1, s_2, \dots)\in S' = \lim_n S_n$.
With the same notations from the last paragraph, we have a sequence 
$$s^\circ=(s_0^\circ,s_1^\circ, s_2^\circ,\cdots)\in S',$$
where each $s_n^\circ$ is a center point in step $n$ corresponding to the unique component that hosts $s_{n+1}$ in $S_{n+1}$.
Consider the images of $s'$ and $s^\circ$ in $\prod_n S_0$: 
the image of $s'$ is $(s_0,s_0,s_0,\dots )$; we denote the image of  $s^\circ$ by
$$\bar s^\circ:=(\bar s_0^\circ, \bar s_1^\circ, \bar s_2^\circ,\cdots).$$
Let $f_n:S_n\ra N$, $f_{n,M}:E_n\ra M$ and $\pi_n: S_n\ra S_0$ be the natural projections.
Then $f_n=f\circ \pi_n$.
Since $s_n \in V_{n} \subset p_n(U_{n})$, there exists some $e_n \in U_{n}$ with 
$$p_n(e_n) = s_n.$$
Because the projection $U_{n,M}$ of $U_{n}$ to $M$ is contained in $x_n + W_n$, we have $f_{n,M}(e_n)\in x_n + W_n$.
After applying $\alpha$ we get
$$\alpha(f_{n,M}(e_n)) \in \alpha(x_n) + \alpha(W_n).$$
The commutativity of the square
$$\xymatrix{
E_n\ar[r]^{f_{n,M}}\ar[d]_{p_n}&M\ar[d]^{\alpha}\ar[d]\\
S_n\ar[r]^{f_n}&N
}$$
gives
$$f(\bar s_n^\circ)=\alpha(x_n) 
\text{\quad and \quad}
f(s_0)=\alpha(f_{n,M}(e_n)) 
$$
for all $n$.
Therefore
$$f(s_0) \in f(\bar s_n^\circ) + \alpha(W_n).$$
Since $W_n$ is symmetric and $\alpha$ is additive,
$$f(\bar s_n^\circ) \in f(s_0) + \alpha(W_n).$$
When $n\ra \infty$,  $\alpha(W_n)\ra 0$, hence 
$$ \lim_{n \to \infty} f(\bar s_n^\circ) = f(s_0).$$
Since $\alpha$ is continuous, 
$$\alpha(g(s')) = \lim_{n \to \infty} \alpha(x_n) = \lim_{n \to \infty} f(\bar s_n^\circ) = f(s_0)=f(\pi(s')).$$

Now we show the map $g: S' \to M$ is continuous. 
Let $O\subset M$ be an open subset which contains a point $x $ from the image of $g$ (otherwise there is nothing to show).
Let $s'=(s_0,s_1,\dots)\in S'$ be a preimage of $x$.
We will show that $s'$ has an open neighborhood contained in $g^{-1}(O)$.
Since $O$ is an open neighborhood of $x$ and $\{W_k\}$ forms a basis of neighborhoods of $0$, there exists an integer $n_0 \ge 3$ such that 
$$x + W_{n_0-3} \subset O.$$
By taking the limit $m \to \infty$ in the Cauchy relation \eqref{ColimCok_Cauchy}, we get
$$x - x_n \in \bar{W_{n-1}} .$$
Since $\bar{W_{n-1}} \subset W_{n-1} + W_{n-1} $ by  \cite[Proposition 1.4.4]{AT_TopGp}, and 
$ W_{n-1} + W_{n-1}\subset W_{n-2}$ by our choice of $W_n$'s, we obtain
$$x \in  x_n + W_{n-2}
\text{\quad  for all $n \ge 2$.}$$
In particular, $ x \in x_{n_0}+ W_{n_0-2}$. 
Because $W_{n_0-2}$ is symmetric, 
$x_{n_0} \in x + W_{n_0-2}$.
Consequently, 
$$x_{n_0}+ W_{n_0-2} \subset x+ W_{n_0-2} + W_{n_0-2} \subset x + W_{n_0-3} \subset O.$$
Let $V_{n_0} \subset S_{n_0}$ be the clopen component containing $s_{n_0}$. 
Consider the canonical projection $S'\ra S_{n_0+1}$.
Since this map is continuous and $S_{n_0+1}$ is equipped with the disjoint union topology, the preimage of the component corresponding to $V_{n_0}$, say $V'$, is a clopen subset of $S'$.
By construction, $s' \in V'$.
\begin{center}
\textit{We claim that $g(V')\subset O$.}
\end{center}
This will establish the desired continuity.

It remains to prove the claim.
Namely, we want to show that $g(t')\in O$ for any $t'=(t_0,t_1,\dots)\in V'$.
We can run the same analysis on $t'$ as we did for $s'$.
By the definition of $V'$, the projection $t_{n_0+1}$ of $t'$ to $S_{n_0+1}$ belongs to the exact same component over $V_{n_0}$ as $s_{n_0+1}$. 
This implies that for the sequence $t'$, the chosen clopen neighborhood $V_n(t')\subset S_n$ of $t'$ are identical to that of $s'$ till the step $n_0$, namely 
\begin{center}
$V_n(t')=V_n$ for all $n\le n_0$.
\end{center}
In particular, the projection $x_{n_0}(t')$ to $M$ is exactly $x_{n_0}$.
Just as for $s'$, we have
$$g(t') \in x_{n_0}(t') + W_{n_0-2} = x_{n_0}+ W_{n_0-2}.$$
Since we established $x_{n_0} + W_{n_0-2} \subset O$, it follows directly that $g(t') \in O$.
Therefore $g(V')\subset O$ as claimed.
\end{proof}

\begin{lem}\label{ColimOpenImm}
Let $I$ be a filtered category and $``\colim_{i\in I}" M_i$ be a diagram in $\TopMod(A)$ such that every transition map $M_i\to M_j$ is an open embedding.
Then the natural map
$$ \colim_{i\in I} \ul{M_i} \ra \ul{\colim_{i\in I} M_i} $$
is an isomorphism in $\Cond(A)$.
\end{lem}
\begin{proof}
Let $M = \colim_{i\in I} M_i$ be the colimit in $\TopMod(A)$.
Algebraically, $M$ is the colimit of the underlying $A$-modules.
Topologically, $M$ is equipped with the colimit topology, meaning a subset $U \subset M$ is open if and only if its preimage in $M_i$ is open for all $i \in I$.
Since all transition maps $M_i \to M_j$ are open immersions, each canonical map $M_i \to M$ is also an open immersion.
In particular, $\{M_i\}_{i\in I}$ forms an open cover of $M$.

To show that $\colim_{i\in I} \ul{M_i} \ra \ul{M}$ is an isomorphism in $\Cond(A)$, we evaluate both sides on an arbitrary profinite set $S$.
Colimits in $\Cond(A)$ are computed by sheafifying the presheaf colimit. We claim that the presheaf colimit is already a sheaf and is isomorphic to $\ul{M}$.
That is, we will show that the natural map of sets
$$ \colim_{i\in I} \Hom_{\Top}(S, M_i) \ra \Hom_{\Top}(S, M) $$
is a bijection.

\textit{Injectivity.}
Suppose two continuous maps $f, g: S \to M_i$ become equal in $M$. Since the transition maps are open immersions, they are in particular injective. Hence $M_i \to M$ is injective, which implies $f = g$ already in $M_i$.

\textit{Surjectivity.} 
Let $f: S \to M$ be a continuous map.
Since $S$ is compact, its continuous image $f(S)$ is a compact subset of $M$.
Because $\{M_i\}_{i\in I}$ is an open cover of $M$, the compact set $f(S)$ is covered by finitely many $M_{i_1}, \dots, M_n$.
Since $I$ is filtered, there exists some $k \in I$ such that all $M_{i_1}, \dots, M_n$ map into $M_k$.
Therefore, $f(S) \subset M_k$, which means $f$ factors through a continuous map $S \to M_k$.
This proves the surjectivity, and hence the lemma.
\end{proof}

\begin{lem}\label{IndProExact0Lem}
The restricted functor
$\TopMod(A)^\delta\subset\TopMod(A)^{\cgwh}\xra{\Phi_2} \Cond(A)$
to the full subcategory of discrete $A$-modules is exact.
\end{lem}
\begin{proof}
This is clear, as the restricted $\Phi_2$ factors through $\LCA$, and both $\TopMod(A)^\delta\ra \LCA$ and $\LCA\ra \Cond(A)$ are strict exact functors between quasiabelian categories. See \cite[the last paragraph on p. 26]{CS}
\end{proof}

\begin{lem}\label{noneedshifify}
Fix an uncountable strong limit cardinal $\kappa.$
Then for any $N\in \Top,$
$\Hom_{\Top}(-,N): \Prof_{\kappa}\ra \Set$ 
is a sheaf.
\end{lem}
\begin{proof}
This is clear from \cite[p. 12, remarks after the proof of Proposition 2.7]{CS} and \cite[Example 1.5]{CS}.
\end{proof}

\begin{thm}\label{thm1}
The composite functor
$$\Pro_{\N}\Mod(A)\xra{\Phi_1} \TopMod(A)^{\cgwh}\xra{\Phi_2}\Cond(A), \quad ``\lim_\N" M_i\mapsto \ul{\lim_\N M_i}$$
factors through $\Solid(A)$.
We denote this functor by
$$\Phi:\Pro_{\N}\Mod(A)\ra \Solid(A)$$
and it is left exact.
Moreover, $\Phi$ induces a fully faithful and exact functor
$$\Phi^{\ML}:\Pro_{\N}^{\ML}\Mod(A)\ra \Solid(A).$$
\end{thm}
\begin{proof}
$\Phi$ is left exact by \Cref{Phi1commlim}, \Cref{Phi2FF} and the fact that limits in $\Solid(A)$  are also limits in $\Cond(A)$ (\cite[Proposition 7.5(i)]{CS}).
The fully faithfulness of $\Phi^{\ML}$ follows from \Cref{Phi1FF} and \Cref{Phi2FF}.
It is also right exact, because the Mittag-Leffler condition guarantees that the limit of condensed presheaves preserves surjections, and a further sheafification step will not destroy this.
It remains to show $\Phi$ factors through $\Solid(A)$.
Since discrete modules are solid, $\ul{M_i}\in \Solid(A)$ for all $i$.
Then $\ul{\lim_\N M_i}\stackrel{\ref{Phi2FF}}{=}\lim_\N \ul{M_i}\in \Solid(A)$, as $\Solid(A)\subset \Cond(A)$ is closed under all limits.
\end{proof}
\begin{rmk}
In \cite{MPV_TateCondv2}, the authors showed a result similar to our result on
$\Phi^{\ML}:\Pro^{\ML}_\N \Mod(A)\ra \Solid(A)$.
In short, they have shown that the category of $\aleph_0$-Tate objects of finitely generated projective $A$-modules fully faithfully embeds into $\Solid(A)$, and this embedding functor is also exact. 
The definition of Tate objects they use assumes that all the transition maps are surjective.
As the main application we have in mind is Deligne's $j_!$ for coherent sheaves, which always produces pro-objects with non-surjective (and also non-Mittag-Leffler) transition maps, we have to include a study on non-Mittag-Leffler pro-objects.
\end{rmk}

\begin{defn}\label{RPhiDef}
Define
$$R\Phi: \Pro_{\N}\Mod(A)\ra D(\Solid(A)),\quad ``\lim_\N" M_i\mapsto R\lim_\N \ul{M_i}.$$
\end{defn}

\section{Compatibility with $j_!$}

Before we proceed to the main theorem, we need some flatness statements.
\begin{prop}\label{flatmodvsflatcondmod}
Let $A$ be a noetherian ring. $f\in A$ is a nonzerodivisor.
Let $\hat A:=\lim_n A/f^n$, and we equip it with the inverse limit topology, where all $A/f^n$ are discrete. 
Then $\ul A\ra \ul{\hat A}$ is flat, i.e., $-\otimes_{\ul A} \ul {\hat A}:\Cond(A)\ra \Cond(A)$ is exact.
\end{prop}

\begin{proof}
Fix an  uncountable strong limit cardinal $\kappa$ and let $S=\lim_{i\in I} S_i\in \Proj_\kappa$, where all $S_i$ are finite.
Since every profinite space is the inverse limit of its discrete quotient spaces (\cite[Proposition 1.1.7]{Wilson}), by possibly replacing the pro-system $\{S_i\}_{i\in I}$ with the pro-system of the finite quotients of $S$, we can assume that all transition maps are surjective.
We want to show that $\ul A(S)\ra \ul{\hat A}(S)$ is a flat ring map.

Since maps from a compact space to a discrete space are continuous if and only if they are locally constant, the ring $B := \ul{A}(S)$ is precisely the ring of locally constant functions.
To understand the module structure of $\widehat{B}:= \ul{\widehat{A}}(S)$, we apply a result of Specker and N\"obeling. 
By \cite[Theorem 5.4]{CS}, the group of continuous maps $\Hom_{\Top}(S, \Z)$ is a free abelian group. Let $E$ be a basis of $\Hom_{\Top}(S, \Z)$. For any discrete ring $R$, we have a canonical isomorphism $\Hom_{\Top}(S, R) \cong \Hom_{\Top}(S, \Z) \otimes_{\Z} R \cong \bigoplus_{e \in E} R$. 
Applying this to $R = A/f^n$, we obtain
\[
\widehat{B} := \ul{\widehat{A}}(S) = \lim_n \Hom_{\Top}(S, A/f^n) \cong \lim_n \bigoplus_{e \in E} A/f^n.
\]
This shows that $\widehat{B}$ is the $f$-adic completion of the free $\widehat{A}$-module $F := \bigoplus_{e \in E} \widehat{A}$. 
(Note that we did not claim any commutativity of $\lim_n$ with $\bigoplus_{e\in E}$.)
Since $A$ is noetherian, $\widehat{A}$ is noetherian. 
By \cite[\href{https://stacks.math.columbia.edu/tag/06LE}{Tag 06LE}]{stacks-project},
$\widehat{B}$ is flat over $\widehat{A}$. 
Since $A \to \widehat{A}$ is flat (\cite[Thm 8.8]{Matsumura}), the composition $A \to \widehat{A} \to \widehat{B}$ is flat.

Now we show that $B \to \widehat{B}$ is flat using the equational criterion of flatness (\cite[\href{https://stacks.math.columbia.edu/tag/00HK}{Tag 00HK}]{stacks-project}). Suppose we are given elements $c_l \in B$ and $x_l \in \widehat{B}$ ($l = 1, \dots, m$) satisfying the relation
\eq{flat1}{
\sum_{l=1}^m c_l x_l = 0 \quad \text{in } \widehat{B}.
}
Since there are only finitely many $c_l$, there exists a sufficiently large $i \in I$ such that all $c_l$ factor through $S_i$. This is equivalent to saying there is a finite partition of $S$ into clopen subsets $S = \bigsqcup_{j=1}^k U_j$ such that each $c_l$ is constant on each $U_j$. Let $\gamma_{l,j} \in A$ be the constant value of $c_l$ on $U_j$. Let $1_{U_j} \in B$ be the indicator function of $U_j$. 

Multiplying the relation \eqref{flat1} by $1_{U_j}$, we get $\sum_{l=1}^m c_l 1_{U_j} x_l = 0$ in $\widehat{B}$. Since $c_l 1_{U_j} = \gamma_{l,j} 1_{U_j}$, this becomes
\eq{flat2}{
\sum_{l=1}^m \gamma_{l,j} (1_{U_j} x_l) = 0.
}
This is a linear relation among the elements $1_{U_j} x_l \in \widehat{B}$ with coefficients $\gamma_{l,j} \in A$. Because we have already proved that $\widehat{B}$ is flat over $A$, the equational criterion implies that there exist elements $y_{j,h} \in \widehat{B}$ and coefficients $\alpha_{l,j,h} \in A$ such that
\[
1_{U_j} x_l = \sum_h \alpha_{l,j,h} y_{j,h} \quad \text{and} \quad \sum_{l=1}^m \gamma_{l,j} \alpha_{l,j,h} = 0
\]
for all $j$ and $h$. Now, define $a_{l, (j,h)} := \alpha_{l,j,h} 1_{U_j} \in B$. We can reconstruct $x_l$ as a finite sum:
\[
x_l = 1 \cdot x_l = \sum_{j=1}^k 1_{U_j} x_l = \sum_{j=1}^k \sum_h a_{l, (j,h)} y_{j,h}.
\]
Furthermore, the coefficients of \eqref{flat1} satisfy
\[
\sum_{l=1}^m c_l a_{l, (j,h)} = \sum_{l=1}^m c_l \alpha_{l,j,h} 1_{U_j} = \sum_{l=1}^m \gamma_{l,j} \alpha_{l,j,h} 1_{U_j} = \left( \sum_{l=1}^m \gamma_{l,j} \alpha_{l,j,h} \right) 1_{U_j} 
\stackrel{\eqref{flat2}}{=} 0.
\]
This verifies the equational criterion for the map $B \to \widehat{B}$. 
Hence $B \to \widehat{B}$ is flat.
\end{proof}

\begin{lem}\label{locflat}
Let $A$ be a ring, $T\subset A$ be a multiplicatively closed subset.
Then map $\ul A\ra \ul{T^{-1}A}$ is flat, i.e., $-\otimes_{\ul A} \ul{T^{-1}A}:\Cond(A)\ra \Cond(A)$ is exact.
\end{lem}
\begin{proof}
We provide two proofs for this.

(First proof)
Because $T$ is multiplicative, the colimit $ T^{-1}A = \colim_{t \in T} A $ is filtered.
Since $\ul{(-)}$ preserves filtered colimits of discrete modules (by replacing the filtered inductive system by its image in the colimit, we achieve a new filtered inductive system with the same colimit and of which all transition maps are open embeddings. Hence this is a special case of \Cref{ColimOpenImm}), we get
$$\ul{T^{-1}A}=\colim_{t\in T} \ul A.$$
Because tensor product commutes with all colimits, we conclude that $-\otimes_{\ul A} \ul{T^{-1}A}$ is exact.

(Second proof)
For any $S \in \Prof_\kappa$, we claim that
$$T^{-1}(\ul A(S))=\ul{T^{-1}A}(S).$$
\textit{Surjectivity.}
Because  $T^{-1}A$ is discrete, any map $f$ from a compact space $S$ will only have a finite image. 
Let $\{\frac{a_1}{t_1},\dots,\frac{a_n}{t_n}\}$ be the image of $f$.
Let $t=t_1\dots t_n$.
This is still in $T$ because $T$ is multiplicatively closed.
We can rewrite the image set of $f$ as $\{\frac{b_1}{t},\dots,\frac{b_n}{t}\}$.
Let $g:=tf$. 
Maps from a compact space to a discrete space are continuous if and only if they are locally constant.
As $f$ is continuous, it is locally constant.
As $g$ is a scalar of $f$, it is also locally constant. 
Therefore $g$ is continuous.
This means $\frac{g}{t}$ is a preimage of $f$. This proves the surjectivity.

\noindent
\textit{Injectivity.}
Let $g:S\ra A$ be a continuous function such that $\frac{g(s)}{t}=0$ for all $s\in S$. 
Since $g$ is a continuous function from a compact space to a discrete space, it  has a finite image, say $\{b_1,\dots, b_n\}$.
Then $\{\frac{b_1}{t},\dots, \frac{b_n}{t}\}$ is the image set of $\frac{g}{t}$.
As $\frac{g}{t}$ is the zero map, there exist $t_1,\dots,t_n \in T$ such that $t_ib_i=0$ for all $i=1,\dots,n$.
Let $t':=t_1\dots t_n$. Then $t'g(s)=0$ for all $s\in S$.
This means $\frac{g}{t}$ is the zero map, namely $\frac{g}{t}=0$ in $T^{-1}(\ul A(S))$. 
\end{proof}

To simplify notations, from now on we omit the ``derived'' part from the notations, except for the tensor product (e.g. $\otimes_{\ul A}$, $\otimes_{A_\bbox}$, $\otimes_{(A[T],A)_\bbox}$, etc.) and $R\lim$. All the other functors in this section  shall be read as derived, such as $\Phi$ refers to $R\Phi$, $f^*$ refers to $Lf^*$, etc..

\begin{thm}\label{thm2}
Let $A$ be a finite type $\Z$-algebra, and $f\in A$ be a nonzerodivisor. Let $j:\Spec A[\frac{1}{f}]\hra \Spec A$ be the open immersion.
Then the natural functor 
$$\Phi:\Pro_{\N}\Mod(A)\subset D(\Pro_{\N}\Mod(A)) \ra D(\Solid(A))$$
given in \Cref{RPhiDef} is compatible with the $j_!$ defined by Deligne \cite[Appendix]{Hartshorne_RD} and the one defined by Clausen-Scholze \cite[Lecture VIII]{CS} for finitely presented modules.
That is, 
the diagram
\eq{thmeq1}{\xymatrix{
\Mod^{\fp}(A[\frac{1}{f}])\ar[r]^{\Phi}\ar[d]^{j^{\De}_!}&
D(\Solid(A[\frac{1}{f}]))\ar[d]^{j^{\CS}_!}\\
\Pro_{\N}\Mod^{\fp}(A)\ar[r]^{\Phi}&D(\Solid(A))
}}
is commutative.
\end{thm}
\begin{proof}
To show that the diagram commutes, we will construct a natural transformation
$$\eta: \Phi j_!^{\De} \to j_!^{\CS} \Phi$$ 
on $\Mod^{\fp}(A[\frac{1}{f}])$ 
and show it is a natural isomorphism.
To this end, we first describe the two functors $\Phi j_!^{\De}$ and $j_!^{\CS} \Phi$ at a given module $M$.

Take $M\in \Mod^{\fp}(A[\frac{1}{f}])$.
Then as $A$ is noetherian, there exists an $A$-submodule $M'\in \Mod^{\fp}(A)$ of $M$, such that 
$$j^*M' =M'\otimes_A A[\frac{1}{f}]= M$$
in $\Mod(A[\frac{1}{f}]).$
(In fact: since $M$ is the colimit of its finitely generated $A$-submodules, there exists one single $A$-submodule such that it contains all the finitely many generators of $M$ as an $A[\frac{1}{f}]$-module.
As $A$ is noetherian, any finitely generated $A$-module is automatically finitely presented.)

For each $n$, $0\ra f^n M'\ra M'\ra M'/f^nM'\ra 0$ is an exact sequence of discrete $A$-modules.
Since the functor assigning a condensed module to a discrete module is exact 
(as discrete modules form a full exact subcategory of locally compact abelian groups and hence a full exact subcategory of condensed modules, see \cite[Corollary 4.9]{CS}), the sequence
$$0\ra \ul{f^n M'}\ra \ul M'\ra \ul{M'/f^nM'}\ra 0$$
is exact in $\Cond(A)$. Therefore
$$R\lim_n \ul{f^n M'}\ra \ul M'\ra R\lim_n \ul{M'/f^nM'}=\lim_n \ul{M'/f^nM'}\xra{+1} $$
is a distinguished triangle.
Hence we have
\eq{thmeqDe2}{\Phi(j_!^{\De}M)= \Phi(``\lim_n"f^nM')=R\lim_n \ul{f^n M'}
=
{\rm Cone}(\ul M'\ra \lim_n \ul{M'/f^nM'})[-1].}

The image of the schematic morphism $j$  in the category of adic spaces (via Huber's fully faithful functor $t$, see \cite[Proposition 4.2(i)]{Huber_gen}), factors canonically as a composition of two morphisms of adic spaces, namely
$$t(\Spec A[\frac{1}{f}])=\Spa(A[\frac{1}{f}],A[\frac{1}{f}])
\xra{j_1}
\Spa(A[\frac{1}{f}],A)
\xra{j_2}
\Spa(A,A)=t(\Spec A).$$
Hence
$j_!^{\CS}$ factors as
$$D(A[\frac{1}{f}]_\bbox)
\xra{j_{1!}}
D((A[\frac{1}{f}],A)_\bbox)
\xra{j_{2!}}
D(A_\bbox),
$$
where $j_{2!}=j_{2*}$ is the forgetful functor.
$j_{1!}$ and $j_{2!}$ fit into the following commutative diagram
$$\xymatrix{
D(A[T]_\bbox)\ar[r]^{j'_{1!}}\ar@/^5ex/[rr]^{j'_!}
&
D((A[T],A)_\bbox)
\ar[r]^(0.6){j'_{2!}}
&
D(A_\bbox)
\\
D(A[\frac{1}{f}]_\bbox)
\ar[u]^{h'_!}
\ar[r]^{j_{1!}}\ar@/_4ex/[rr]_{j^{\CS}_!}&
D((A[\frac{1}{f}],A)_\bbox)
\ar[u]_{h_!}\ar[r]^(0.6){j_{2!}}
&
D(A_\bbox).\ar@{=}[u]
}$$
Here
$j_1'$ is the canonical map $\Spa(A[T],A[T])\ra \Spa(A[T],A)$, and $j'_{1!}$ is given by $$(-)\otimes^L_{(A[T],A)_\bbox} (\ul{A((T^{-1}))}/\ul{A[T]})[-1],$$
as introduced in the proof of \cite[Theorem 8.13(ii)]{CS}, and $A((T^{-1}))$ is regarded as a topological $A[T]$-module with the colimit topology of the $T^{-1}$-adic topology of $A[[T^{-1}]]$.
The vertical maps, $h_!,h'_!$ are induced from $A[T]\ra A[\frac{1}{f}]$, mapping $T$ to $1/f$.
As this ring map is surjective, $A[\frac{1}{f}]_\bbox=(A[\frac{1}{f}],A[T])_\bbox.$
In particular, $h_!=h_*$ and $h'_!=h'_*$, and hence they both are forgetful functors.
We have
$$j^{\CS}_{!}=j_{2!}j_{1!}= j'_{2!}j'_{1!}h'_!.$$

Let  $M\in \Mod^{\fp}(A[\frac{1}{f}])$, $M'\in \Mod(A)$ be as before. 
Before we proceed, we discuss some algebraic and topological properties of $f$-adic completions and localizations, and their compatibility with the condensed setting.
Let $\hat M' := \lim_n M'/f^nM'$ be the $f$-adic completion of $M'$ equipped with the inverse limit topology.
By \Cref{Phi2FF}, we have 
$$\ul{\hat M'} = \lim_n \ul{M'/f^nM'}.$$
Let $\hat M'[\frac{1}{f}]$ be the algebraic localization of $\hat M'$, and we equip it with the colimit topology of the inductive system of topological modules $\hat M'\xra{f}\hat M'\xra{f}\cdots$.
Because $M$ is an $A[\frac{1}{f}]$-module, it has no $f$-torsion.
Since $M'$ is an $A$-submodule of $M$, it has no $f$-torsion  as well. 
That is, $M'\xra{f} M'$ is injective. 
Since $A$ is noetherian, $\hat A$ is flat over $A$.
The base change of this map along the flat map $A\ra \hat A$ is injective.
Since $M'$ is finitely presented as an $A$-module, $\hat M'=M'\otimes_A \hat A$.
That is, $\hat M'$ is $f$-torsionfree.
Hence the transition maps in the inductive system $\hat M' \xrightarrow{f} \hat M' \xrightarrow{f} \cdots$ are isomorphisms onto the open subgroup $f\hat M'$ in the codomain.
Thus, $\hat M'[\frac{1}{f}]$ is the topological colimit of a sequence of open embeddings. 
By \Cref{ColimOpenImm},
the functor $\ul{(-)}: \TopMod(A) \to \Cond(A)$ preserves such filtered colimits, and hence gives
$$\ul{\hat M'}[\frac{1}{f}] \simeq \ul{\hat M'[\frac{1}{f}]}.$$
To conclude, we have canonical isomorphisms in $\Cond(A)$ which we will be using freely:
\eq{thmeqCSblue}{ (\lim_n \ul{M'/f^nM'})[\frac{1}{f}] = \ul{\hat M'}[\frac{1}{f}]= \ul{\hat M'[\frac{1}{f}]}. }

We continue with the computation of $j_!^{\CS}(\Phi(M))$.
Since $\ul{A[T]}$ is a projective object in $\Solid((A[T],A)_\bbox)$, we have
\[\ul{A[\frac{1}{f}]}\otimes^L_{(A[T],A)_\bbox} \ul{A[T]}
=\ul{A[\frac{1}{f}]}\otimes_{(A[T],A)_\bbox} \ul{A[T]}
=\ul{A[\frac{1}{f}]}.
\numberthis \label{thmeqCS1}\]
Since 
$$0\ra \ul{A[T]}\xra{fT-1} \ul{A[T]}\xra{T\mapsto 1/f} \ul{A[\frac{1}{f}]}\ra 0$$
is a projective resolution of $\ul{A[\frac{1}{f}]}$ in $\Solid((A[T],A)_\bbox)$,
we can compute:
\begin{align*} 
\numberthis \label{thmeqCS2}
\ul{A[\frac{1}{f}]}\otimes^L_{(A[T],A)_\bbox} \ul{A((T^{-1}))}
&=\left(\ul{A((T^{-1}))}
\xra{T^{-1}-f}   \ul{A((T^{-1}))}\right)
\\&
=\ul{A((T^{-1}))/(T^{-1}-f)}
\\&
=\ul{\left(A[[T^{-1}]]/(T^{-1}-f)\right)[\frac{1}{f}]}
\\&
=\ul{\hat A[\frac{1}{f}]}.
\end{align*}
We explain the reasoning at each step:
\begin{itemize}
\item 
In the first equality, we have used the given projective resolution of $\ul{A[\frac{1}{f}]}$, the identification of the bifunctor $-\otimes_{(A[T],A)_\bbox}-=(-\otimes_{\ul{A[T]}}-)\otimes_{\ul{A[T]}}A_\bbox$ and
$\ul{A((T^{-1}))}\in \Solid((A[T],A)_\bbox)$. 
\item 
In the second equality, we have used the following short exact sequence in $\Cond(A)$:
$$0\ra \ul{A((T^{-1}))}\xra{T^{-1}-f}\ul{A((T^{-1}))}\ra \ul{A((T^{-1}))/(T^{-1}-f)} \ra 0.$$
The corresponding algebraic sequence is clearly exact as modules. 
Since $T^{-1}-f$ is a closed embedding, it is a strict exact sequence in the category of topological modules.
Since $\ul{(-)}$ commutes with limits, the condensed sequence is left exact. 
Since $A((T^{-1}))$ is a  colimit of $A[[T^{-1}]]$ via open embeddings, and $A[[T^{-1}]]$ is sequential complete, $T_1$ and first countable, $A((T^{-1}))$ satisfy these properties as well. 
By \Cref{ColimCok}, the condensed sequence is right exact.
\item 
The rest are purely algebraic: localizations commute with quotients, and for finitely generated modules over a noetherian ring, completion commutes with quotients by \cite[Theorem 8.1(ii)]{Matsumura}.
\end{itemize}
Because 
$$0\ra \ul{A[T]}\ra \ul{A((T^{-1}))}\ra \ul{A((T^{-1}))}/\ul{A[T]}\ra 0 $$
is naturally exact, we can replace $\ul{A((T^{-1}))}/\ul{A[T]}$ with the two term complex $\ul{A[T]}\ra \ul{A((T^{-1}))}$ in the derived category and compute
\begin{align*}
\numberthis \label{thmeqCS3}
&\ul{A[\frac{1}{f}]}\otimes^L_{(A[T],A)_\bbox} (\ul{A((T^{-1}))}/\ul{A[T]})
\\
&={\rm Cone}\left(
\ul{A[\frac{1}{f}]}\otimes^L_{(A[T],A)_\bbox} \ul{A[T]}
\xra{}
\ul{A[\frac{1}{f}]}\otimes^L_{(A[T],A)_\bbox} \ul{A((T^{-1}))}
\right)
\\
&
={\rm Cone}\left( \ul{A[\frac{1}{f}]}\ra \ul{\hat A[\frac{1}{f}]}\right)
\quad\text{(by \eqref{thmeqCS1} and \eqref{thmeqCS2})}.
\end{align*}
Moreover, since $M'\in \Mod^{\fp}(A)$, we have
\begin{align*} 
\numberthis \label{thmeqCS4}
\ul M'\otimes^L_{A_\bbox} \ul{\hat A[\frac{1}{f}]}
&=\ul M'\otimes^L_{A_\bbox} \ul{\hat A} \otimes^L_{A_\bbox}\ul{A[\frac{1}{f}]}
\\
&=\left(\left(\ul M'\otimes^L_{\ul A} \ul{\hat A} \right)\otimes^L_{\ul A} A_\bbox\right)
\otimes^L_{A_\bbox}\ul{A[\frac{1}{f}]}
\\
&=\left(\left(\ul M'\otimes_{\ul A} \ul{\hat A} \right)\otimes^L_{\ul A} A_\bbox\right)
\otimes^L_{A_\bbox}\ul{A[\frac{1}{f}]}
\\&=
\left(\ul{\hat M'} \otimes^L_{\ul A} A_\bbox\right)
\otimes^L_{A_\bbox}\ul{A[\frac{1}{f}]}
\\&=
\ul{\hat M'}
\otimes^L_{A_\bbox}\ul{A[\frac{1}{f}]}
\\&=
\ul{\hat M'[\frac{1}{f}]}
\end{align*}
We explain the reasoning at each step:
\begin{itemize}
\item 
The first equality follows from the derived $A$-solidification of the identification
$$\ul{\hat A[\frac{1}{f}]}
=\ul{\hat A}\otimes_{\ul A}\ul{ A[\frac{1}{f}]}
= \ul{\hat A}\otimes^L_{\ul A}\ul{ A[\frac{1}{f}]}.$$
The last equality here comes from the flatness in either the second slot \Cref{locflat}, or the first slot \Cref{flatmodvsflatcondmod}.
\item 
The second equality comes from the identification $-\otimes^L_{A_\bbox}-=(-\otimes^L_{\ul A}-)\otimes^L_{\ul A} A_\bbox$.
This identification follows from the uniqueness statement of $-\otimes^L_{A_\bbox}-$ in \cite[Theorem 7.5(ii)]{CS}.
\item 
The third equality follows from \Cref{flatmodvsflatcondmod}.
\item 
The fourth equality follows from applying \cite[Theorem 8.7]{Matsumura} on each section and then sheafifying.
\item 
The fifth equality comes from the fact that $\hat M'$ is a solid $A$-module, as it is the limit of discrete modules.
\item 
The sixth equality follows from the following identifications:
\begin{align*} 
\hat M'\otimes^L_{A_\bbox}\ul{A[\frac{1}{f}]}
=\left(\ul{\hat M'}\otimes^L_{\ul A}\ul{A[\frac{1}{f}]}\right)\otimes^L_{\ul A} A_\bbox
=\left(\ul{\hat M'}\otimes_{\ul A}\ul{A[\frac{1}{f}]}\right)\otimes^L_{\ul A} A_\bbox
\\
=\ul{\hat M'[\frac{1}{f}]}\otimes^L_{\ul A} A_\bbox
=\ul{\hat M'[\frac{1}{f}]}.
\end{align*}
At the second equality here we have used that $\ul A\ra \ul{A[\frac{1}{f}]}$ is flat (\Cref{locflat}).
The third equality holds because colimits commutes with tensor, and then we use the identification \eqref{thmeqCSblue}.
The last equality holds because $\ul{\hat M'[\frac{1}{f}]}$ is $A$-solid; and this is because it is the colimit of $\hat M'$, which is $A$-solid.
\end{itemize}
Altogether,
\begin{align*} 
\numberthis \label{thmeqCS5}
j_!^{\CS}(\Phi(M))&= 
j_!^{\CS}(\ul M)=
j_!^{\CS}(j^*\ul M')
=\ul M'\otimes^L_{A_\bbox} j^{\CS}_!\ul{A[\frac{1}{f}]}
\quad \text{(\cite[Theorem 8.13(iii)]{CS})}
\\
&=\ul M'\otimes^L_{A_\bbox}\left(\ul{A[\frac{1}{f}]}\otimes^L_{(A[T],A)_\bbox} (\ul{A((T^{-1}))}/\ul{A[T]})\right)[-1]
\\
&=\ul M'\otimes^L_{A_\bbox}
{\rm Cone}\left( \ul{A[\frac{1}{f}]}\ra \ul{\hat A[\frac{1}{f}]}\right)[-1]
\quad\text{(\Cref{thmeqCS3})}
\\
&=\left(\ul M'\otimes^L_{\ul A}
{\rm Cone}\left( \ul{A[\frac{1}{f}]}\ra \ul{\hat A[\frac{1}{f}]}\right)\otimes^L_{\ul A}A_\bbox\right)[-1]
\\
&=
\left({\rm Cone}\left( \ul M'\otimes_{\ul A}\ul{A[\frac{1}{f}]}\ra \ul M'\otimes_{\ul A} \ul{\hat A[\frac{1}{f}]}\right)\otimes^L_{\ul A}A_\bbox\right)[-1]
\text{\quad (\Cref{locflat})}
\\
&={\rm Cone}(\ul M\ra \ul{\hat M'[\frac{1}{f}]})[-1]
\quad\text{(\Cref{thmeqCS4})}
\end{align*}

The computations in \eqref{thmeqDe2} and \eqref{thmeqCS5} yield canonical identifications in $D(\Solid(A))$:
\begin{align*}
\Phi(j_!^{\De}(M)) &\simeq {\rm Cone}\left(\ul M' \to \ul{\hat M'}\right)[-1], \\
j_!^{\CS}(\Phi(M)) &\simeq {\rm Cone}\left(\ul M \to \ul{\hat M'[\frac{1}{f}]}\right)[-1].
\end{align*}
The canonical inclusions $M' \hra M$ and $\hat M' \hra \hat M'[\frac{1}{f}]$ induce a strictly commutative square of topological modules, which upon applying the exact functor $\ul{(-)}$ yields a commutative square in $\Cond(A)$:
\eq{nat_transf-diag}{
\xymatrix{
\ul M' \ar[r] \ar[d] & \ul{\hat M'} \ar[d] \\
\ul M \ar[r] & \ul{\hat M'[\frac{1}{f}]}
}
}
By the functoriality of the cone, this commutative square naturally induces a morphism between the shifted cones:
\eq{nat_transf_eq}{
{\rm Cone}\left(\ul M' \to \ul{\hat M'}\right)[-1] \longrightarrow {\rm Cone}\left(\ul M \to \ul{\hat M'[\frac{1}{f}]}\right)[-1].
}
This induced morphism defines
$$\eta_M:\Phi j_!^{\De}(M) \to j_!^{\CS} \Phi(M).$$
The naturality in $M$ follows from the construction.
Therefore we obtain a natural transformation 
$$\eta:\Phi j_!^{\De} \to j_!^{\CS} \Phi.$$

To conclude the proof, we must show that the morphism \eqref{nat_transf_eq} is an isomorphism. 
For a square-shaped commutative diagram, the morphism between the horizontal cones is an isomorphism if and only if the morphism between the vertical cones is an isomorphism. 
We apply this to diagram \eqref{nat_transf-diag}:
\begin{itemize} 
\item 
The left vertical map $\ul{M'} \to \ul{M}$ is induced by the inclusion of discrete modules $M' \hra M$. 
Since this map is a strict monomorphism between discrete modules, and its cone is simply the underline of the cokernel  of the corresponding discrete modules  by \Cref{IndProExact0Lem}.
\item 
The right vertical map $\ul{\hat M'} \to \ul{\hat M'[\frac{1}{f}]}$ is induced by the inclusion $\hat M' \hra \hat M'[\frac{1}{f}]$. 
Since $\hat M'$ is $f$-torsionfree and $\hat M'[\frac{1}{f}]$ is equipped  with the colimit topology,  this map is an open immersion.
In particular, it is a strict monomorphism.
Since $\ul{(-)}$ commutes with all limits by \Cref{Phi2FF},
$\ul{\hat M'} \to \ul{\hat M'[\frac{1}{f}]}$ is a monomorphism.
Since $\hat M'$ is sequentially complete and first countable, and $\hat M'[\frac{1}{f}]$ is the colimit of $\hat M'$ along open immersions, $\hat M'[\frac{1}{f}]$ is sequentially complete, $T_1$ and first countable as well.
By \Cref{ColimCok} the cokernel of $\ul{\hat M'} \to \ul{\hat M'[\frac{1}{f}]}$ is $\ul{\hat M'[\frac{1}{f}]/\hat M'}$, which also serves as its cone.
\end{itemize}
Therefore, the map between the vertical cones is 
\eq{thmeq2}{ 
\ul{M/M'} \longrightarrow \ul{\hat M'[\frac{1}{f}] / \hat M'}. }
Since both modules are discrete, this is an isomorphism in $\Cond(A)$ if and only if 
$$ M/M'\ra \hat M'[\frac{1}{f}] / \hat M'$$
is an algebraic isomorphism.

For surjectivity, let $z \in \hat M'[\frac{1}{f}]$. We can write $z = y/f^k$ for some $y \in \hat M'$ and $k \ge 0$.
Since the image of $M'$ is dense in $\hat M'$ with respect to the $f$-adic topology, we can write $y = x + f^k y'$ for some $x \in M'$ and $y' \in \hat M'$.
Then $z = (x + f^k y')/f^k = x/f^k + y'$.
Modulo $\hat M'$, the element $z$ is equivalent to $x/f^k$. Since $x/f^k \in M$, it lies in the image of \eqref{thmeq2}, which proves the surjectivity.

For injectivity, suppose $m \in M$ maps to $0$ in $\hat M'[\frac{1}{f}] / \hat M'$.
This means $m \in \hat M'$ inside $\hat M'[\frac{1}{f}]$.
Since $m \in M$, we can write $m = x/f^k$ for some $x \in M'$ and $k \ge 0$.
Since $m \in \hat M'$, we have $x \in f^k \hat M'$.
Thus $x \in M' \cap f^k \hat M'$.
Because $\hat M'/f^k \hat M' \cong M'/f^k M'$, the preimage of $f^k \hat M'$ under the natural map $M' \to \hat M'$ is exactly $f^k M'$.
Hence $x = f^k x'$ for some $x' \in M'$.
This gives $m = x/f^k = x' \in M'$, meaning $m$ represents $0$ in $M/M'$. This proves the injectivity.

This means that \eqref{thmeqDe2} and \eqref{thmeqCS5} agree, and we are done with the proof of the theorem.

\end{proof}

\bibliographystyle{alpha}
\bibliography{Cond-DRW}
\end{document}